\documentclass[12pt,reqno]{amsart}
    
\usepackage{times}
\usepackage{amsmath,amsthm,amsfonts,amssymb}
\usepackage[matrix,frame]{xy}
\usepackage{fullpage}
\usepackage{xcolor}

\newtheorem{theorem}{Theorem}[section]

\newtheorem{definitionn}[theorem]{Definition} 
\newtheorem{lemma}[theorem]{Lemma}
\newtheorem{proposition}[theorem]{Proposition}
\newtheorem{conjecture}[theorem]{Conjecture} 
{ \theoremstyle{remark}\newtheorem*{remark}{Remark} }
\numberwithin{theorem}{section} 
\numberwithin{equation}{section}

\newcommand{\myop}[1]{\operatorname{#1}} 
\newcommand{\Tr}{\myop{Tr}} 
\newcommand{\id}{\myop{id}}
\newcommand{\Mor}{\myop{Mor}}

\newcommand{\N}{\mathbb N}
\newcommand{\n}{\mathbb{N}}
\newcommand{\z}{\mathbb{Z}}
\newcommand{\C}{\mathbb C}
\newcommand{\Comp}{\mathbb C} 
\newcommand{\R}{\mathbb R} 
\newcommand{\G}{\mathbb G}
\newcommand{\g}{\mathbb G}
\newcommand{\tor}{\mathbb{T}}
\newcommand{\HS}{\mathrm{HS}}

\newcommand{\norm}[1]{\|#1\|}
\newcommand{\tp}{\mathop{\xymatrix{*+<.7ex>[o][F-]{\scriptstyle \top}} } }

\begin{document}
    
\title{Property RD and hypercontractivity for orthogonal free quantum groups}
    
\author{Michael Brannan} \address{Michael Brannan: DEPARTMENT OF MATHEMATICS, MAILSTOP 3368, TEXAS A$\&$M UNIVERSITY, COLLEGE STATION, TX 77843-3368, USA}
\email{mbrannan@math.tamu.edu}
 
\author{Roland Vergnioux} \address{Roland Vergnioux: Normandie Univ, UNICAEN,
  CNRS, LMNO, 14000 Caen, France} \email{roland.vergnioux@unicaen.fr}

\author{Sang-Gyun Youn} \address{Sang-Gyun Youn: Department of Mathematics Education, Seoul National University, Gwanak-ro 1, Gwanak-gu, Seoul, 08826, Republic of Korea} \email{s.youn@snu.ac.kr}
    

\begin{abstract}
  We prove that the twisted property RD introduced in \cite{BhVoZa15} fails to
  hold for all non Kac type, non amenable orthogonal free quantum groups. In the
  Kac case we revisit property RD, proving an analogue of the $L_p-L_2$
  non-commutative Khintchine inequality for free groups from \cite{RiXu16}.  As
  an application, we give new and improved hypercontractivity and
  ultracontractivity estimates for the generalized heat semigroups on free
  orthogonal quantum groups, both in the Kac and non Kac cases.
\end{abstract}
    
\maketitle

\section{Introduction}

Property RD (Haagerup's inequality) is a fundamental tool in the study of the
reduced $C^*$-algebra of discrete groups, allowing one to control the operator norm
of convolution operators by means of the much simpler $\ell^2$-norm (see
Section~\ref{subsec_intro_RD} for more details). It appeared in the seminal
paper \cite{Ha78} where it was used in conjunction with Haagerup's approximation
property (HAP) to establish the Metric Approximation Property (MAP) for reduced
$C^*$-algebras of free groups.

The definition of Property RD was extended to discrete quantum groups in
\cite{Ve07} and was proved there to be satisfied by Kac type (unimodular)
orthogonal and unitary free quantum groups. In \cite{Br12} a quantum analogue of
the HAP was established for these free quantum groups, thus yielding a proof of
the MAP for the corresponding reduced $C^*$-algebras. Property RD was moreover
used for the study of other aspects of discrete quantum group operator algebras,
see e.g.  \cite{VaVe07,Ve12,Br14,Yo18a}. Interesting connections to Quantum
Information Theory, specific to the quantum framework, were also unveiled in
\cite{BrCo18}.

The definition of Property RD used in \cite{Ve07} can only be satisfied by Kac
type discrete quantum groups. In \cite{BhVoZa15}, the authors give a "twisted"
version of the definition which holds for all (duals of) $q$-deformations of
connected compact semi-simple Lie groups, and give applications to
noncommutative geometry.

\bigskip

Hypercontractivity describes the regularization effect, in terms of $L_p$-norms,
of a given Markov semigroup.  It has been studied extensively since the early
70's, starting with the work of Nelson and Gross \cite{Ne73,Gr72}, and has found
surprising applications in harmonic analysis, information theory and statistical
mechanics. In the case of the Ornstein-Uhlenbeck semigroup on the Clifford
algebra with one generator, the two-point inequality of Bonami, rediscovered by
Gross \cite{Bo70,Gr75}, already has deep applications to (quantum) information
theory \cite{BuReScWo12,GaKeKeRaWo08,KlRe11,KhVi15}.

In the noncommutative framework, hypercontractivity problems for
Orstein-Uhlenbeck-like semigroups emerged from quantum field theory and optimal
times have been obtained in the fermionic case in \cite{Ne73,CaLi93}, using
noncommutative $L_p$-theory. Moving further away from
the commutative situation, hypercontractivity results for free group algebras
were obtained in \cite{Bi97,JuPaPaPeRi15} (with respect to different
semigroups). Note that the connection between hypercontractivity and Property RD
in that case was already noticed by Biane \cite{Bi97}.

The study of hypercontractivity for discrete quantum group algebras was
initiated in \cite{FrHoLeUlZh17}, where a natural analogue of the heat semigroup
on the reduced $C^*$-algebra of orthogonal free quantum groups was studied. In
the Kac case, the authors of \cite{FrHoLeUlZh17} obtain the ultracontractivity
of these semigroups (at all times), as well as hypercontractivity with explicit
upper bounds for the optimal time to contractivity.

\bigskip

In the present article we pursue the study of Property RD for non Kac type
discrete quantum groups. We prove that non Kac and non amenable orthogonal free quantum groups do
not satisfy the property RD introduced in \cite{BhVoZa15}
(Theorem~\ref{thm_RD_fail}). Then we state and prove a weaker RD inequality
(Proposition~\ref{prop_expRD}) which holds for all orthogonal free quantum
groups, and which was already used without proof in \cite{VaVe07} in a slightly
less precise form.

In the second part of the article we continue the study of ultra- and
hypercontractivity for the heat semigroup on free orthogonal quantum groups. We
obtain in particular the first known results in the non Kac case, namely
ultracontractivity with a {\em strictly positive} optimal time
(Proposition~\ref{prop_ultra_non_kac}) and hypercontractivity for large time
(Proposition~\ref{prop_hyper_non_kac}). In the Kac case we sharpen the upper
bound of \cite{FrHoLeUlZh17} for the optimal time to hypercontractivity
(Theorems~\ref{thm1}, \ref{thm2} and \ref{thm3}), using a non-commutative
Khintchine type inequality (Theorem~\ref{thm:Khintchine}). We give as well a
lower bound for the optimal time to hypercontractivity
(Lemma~\ref{lem-necessary}). Motivated by these results, we end the article with
a conjectural formula for the asymptotical behavior of the optimal time to
hypercontractivity when the rank of the free orthogonal quantum group tends to
infinity.

The article is organized as follows. In Section~\ref{sec_prelim} we recall the
necessary preliminaries about compact quantum groups and Property RD on their
duals. Section~\ref{RDfail} is devoted to the study of Property RD on non Kac
type orthogonal free quantum groups. Finally in Section~\ref{sec_applic} we
produce applications to hypercontractivity as described above.

\bigskip

\noindent {\bf Acknowledgments.} M.B. was supported by NSF grant
DMS-1700267. R.V. was partially supported by the ANR project
ANR-19-CE40-0002. S-G.Y. was supported by the Natural Sciences and Engineering
Research Council of Canada and by the National Research Foundation of Korea (NRF) grant funded by the Korea government (MSIT) (No. 2020R1C1C1A01009681). R.V. thanks Holger Reich and the Algebraic Topology
group at the Freie Universit\"at Berlin for their kind hospitality in the
academic year 2019--2020.

\section{preliminaries}
\label{sec_prelim}
    
We assume that the reader is familiar with the basic notation and terminology on
compact and discrete quantum groups.  For details, we refer the reader to the
standard references \cite{Wo98, Ti08, NeTu13, MaVa98}.  In this paper we will
mainly be concerned with the class of free orthogonal quantum groups and their
associated dual discrete quantum groups.  We now recall these objects.

\subsection{Compact quantum groups}
    
A compact quantum group $\g$ is given by a Woronowicz $C^*$-algebra $C(\g)$,
which is in particular a unital Hopf-$C^*$-algebra with co-associative coproduct
$\Delta:C(\g) \to C(\g) \otimes C(\g)$.  We denote by $h$ the Haar state on
$C(\g)$, which is the unique state on $C(\g)$ satisfying
\[
(h \otimes \id) \Delta = (\id \otimes h)\Delta = h(\cdot) 1.
\]The Haar state induces the inner product $\langle f,g \rangle = h(f^*g)$ and
the norm $\|f\|_2 = h(f^*f)^{1/2}$ for $f$, $g \in C(\g)$. By completion we
obtain the GNS space $L_2(\g)$ with canonical cyclic vector $\xi_0$, and we
denote $\pi_h : C(\g) \to B(L_2(\g))$ the associated representation. The image
of $\pi_h$ is the reduced Woronowicz $C^*$-algebra denoted $C_r(\g)$ and the
associated von Neumann algebra is
$L_{\infty}(\g) = C_r(\g)'' \subset B(L_2(\g))$.

We then define $L_1(\g)$ as the predual of $L_{\infty}(\g)$ and consider the
natural embedding $L_{\infty}(\g)\hookrightarrow L_1(\g)$ given by
$x\mapsto h(\,\cdot\, x)$. Then $(L_{\infty}(\g),L_1(\g))$ is a compatible pair
of Banach spaces, which allows one to define the non-commutative $L_p$-spaces
$L_p(\g)=(L_{\infty}(\g),L_1(\g))_{1/p}$ by the complex interpolation method
\cite{Pi03}.  When the Haar state is tracial we have
$\norm{a}_{L_p(\g)}=h(|a|^p)^{1/p}$ for any $1\leq p<\infty$ and
$a\in L_{\infty}(\g)$.
    
\bigskip

A representation of $\g$ on a Hilbert space $H_v$ is an invertible element
$v \in M(K(H_v)\otimes C(\g))$ such that
$(\id\otimes\Delta)(v) = v_{12} v_{13}$, using the leg-numbering notation. Here, $v_{12}=v\otimes 1$ and $v_{13}=\sigma_{23}(v_{12})$ in $M(K(H_v)\otimes C(\g)\otimes C(\g))$, where $\sigma_{23}$ is the unique extension of the $*$-homomorphism on $K(H_v)\otimes C(\g)\otimes C(\g)$ given by $T\otimes a\otimes b\mapsto T\otimes b\otimes a$.

Furthermore, $v$ is called a unitary representation if $v^*v=\text{Id}_{H_v}\otimes 1_{C(\g)}=vv^*$. If
$H_v$ is finite-dimensional and equipped with an orthonormal basis $(e_i)_i$,
the associated matrix elements of $v$ are
$v_{ij} = (e_i^*\otimes\id)v(e_j\otimes\id)$. Then we have
$v = \sum e_ie_j^* \otimes v_{ij}$ and
$\Delta(v_{ij}) = \sum v_{ik}\otimes v_{kj}$.  For two unitary representations
$v \in M(K(H_v)\otimes C(\g))$ and $w \in M(K(H_w)\otimes C(\g))$, the tensor product
representation is $v\tp w = v_{13} w_{23} \in M(K(H_v\otimes H_w)\otimes C(\g))$.
    
Furthermore, we say that $v$ is {\it irreducible} if
$\mathrm{Mor}(v,v):= \left \{T\in B(H_v): v (T\otimes 1)=(T\otimes
  1)v\right\}=\Comp\cdot\id_{H_v}$.
We denote by $\mathrm{Irr}(\g)$ the set of all irreducible unitary
representations of $\g$ up to unitary equivalence. For each
$\alpha \in \mathrm{Irr}(\g)$ we choose $u = u^\alpha \in \alpha$ and denote
$H_\alpha = H_u$ (which is always finite-dimensional). The coefficients of $u^\alpha$ with respect to some orthonormal basis $(e_i)_i \subset H_\alpha$ are denoted
$u^\alpha_{i,j}$. The multiplicity of an irreducible representation $u$ in
another representation $v$ is
$\mathrm{mult}(u\subset v) = \dim \mathrm{Mor}(u,v)$.
    
There is, for each irreducible unitary representation $u$, a uniquely defined positive
element $Q_u \in B(H_u)$ such that $d_u := \Tr(Q_u) = \Tr(Q_u^{-1})$ and such
that the following orthogonality relations hold
\begin{equation} \label{eq:ortho_relations}
  \begin{aligned}[c]
    h(u_{ij}^* u_{kl}) &= d_u^{-1} \delta_{jl} (e_k\mid Q_u^{-1}e_i), \\
    h(u_{kl} u_{ij}^*) &= d_u^{-1} \delta_{ik} (e_j\mid Q_u e_l).
  \end{aligned}
\end{equation}
The number $d_u$ is called the {\it quantum dimension} of $u$, as opposed to the
classical dimension $n_u = \dim H_u$. The compact quantum group $\g$ is said to
be of {\it Kac type} if $Q_\alpha = \id_{H_\alpha}$ for all
$\alpha \in \mathrm{Irr}(\g)$.  This is equivalent to the Haar state $h$ being
tracial.
    
The coefficients $u^\alpha_{i,j}$ of irreducible unitary representations span a dense
subalgebra $\mathcal{O}(\g) \subset C(\g)$ which is a Hopf algebra with respect
to the restriction of the coproduct $\Delta$. We recall that $h$ is faithful on
$\mathcal{O}(\g)$ and we shall identify $\mathcal{O}(\g)$ with its image in
$C_r(\g)$ through the GNS representation $\pi_h$ --- in particular we identify a representation $v$ and its image
$(\id\otimes\pi_h)(v) \in B(H_v)\otimes C_r(\g)$. Note also that
$\mathcal{O}(\g)$ is dense in $L_p(\g)$ for any $1\leq p<\infty$.

\bigskip

A compact quantum group $\g$ is said to be a compact matrix quantum group if
there exists a finite generating subset
$\left \{\alpha_1,\cdots, \alpha_n\right\}$ of $\mathrm{Irr}(\g)$ in the sense
that any irreducible unitary representation $u^{\alpha}$ appears as an
irreducible component of a tensor product representation
$u^{\alpha_{m_1}}\tp u^{\alpha_{m_2}}\tp \cdots \tp u^{\alpha_{m_k}}$ for some
$k\in \n$ and $1\leq m_1,m_2,\cdots, m_k\leq n$. In this case, for any $\alpha$,
the minimal number $k \in \n_0$ required to generate $u^{\alpha}$ as a subrepresentation as above is called the {\it length of
  $\alpha$}, and denoted $|\alpha| = k$. The length at the trivial representation is $0$. We say that a non-zero element $f \in C(\g)$ or $C_r(\g)$ has
length $k$ if it can be written as a linear combination of coefficients
$u^\alpha_{i,j}$ with irreducible representations $\alpha$ of length $k$. We
denote $p_k\in B(L_2(\g))$ the orthogonal projection onto the subspace of
$L_2(\g)$ spanned by elements of length $k$.

\subsection{Dual algebras}
  
Associated to each compact quantum group $\g$ is its dual discrete quantum group
$\widehat{\g}$.  For us the main object of interest will be the algebra
\begin{displaymath}
  \ell_\infty(\widehat\g) = \{a \in \textstyle\prod_{\alpha\in\mathrm{Irr(\g)}}B(H_\alpha) :
  (\|a_\alpha\|)_\alpha ~ \text{bounded}\}
\end{displaymath}
and the subalgebras $c_{00}(\widehat\g)$, $c_0(\widehat\g)$ of sequences with
finite support, resp. converging to $0$.  For each $\alpha\in \mathrm{Irr}(\g)$
we denote $p_\alpha$ the corresponding minimal central projection in any of
these algebras. We use the same notation $p_\alpha$ for the orthogonal
projection onto the subspace of $L_2(\g)$ spanned by the GNS images of the
coefficients $u_{i,j}^\alpha$ --- indeed there is a natural representation of
$c_0(\widehat\g)$ on $L_2(\g)$ which realizes this identification.
    
The algebras $c_0(\widehat\g)$ and $C(\g)$ are related through the ``multiplicative
unitary'' $V = \bigoplus_\alpha u^\alpha \in M(c_0(\widehat\g)\otimes C(\g))$.
We endow $c_0(\widehat\g)$ and $\ell_\infty(\widehat\g)$ with the coproduct
$\hat\Delta$ such that $(\hat\Delta\otimes\id)(V) = V_{13}V_{23}$. By definition
this coproduct is related to the tensor product construction for
representations, more precisely we have, for all $\alpha$, $\beta$,
$\gamma\in\mathrm{Irr}(\g)$, $a\in B(H_\gamma)$ and
$T\in\mathrm{Mor}(\gamma, \alpha\tp\beta)$, the following identity in
$B(H_\gamma, H_\alpha\otimes H_\beta)$:
\begin{equation*}
  (p_\alpha\otimes p_\beta)\hat\Delta(a) T = T a.
\end{equation*}

There is a distinguished weight $\hat h$ on $\ell_\infty(\widehat\g)$, called
the {\it left Haar weight}, given by
\begin{displaymath}
  \hat h(a) = \sum_{\alpha\in \mathrm{Irr}(\g)} d_\alpha 
  \mathrm{Tr}(Q_\alpha a_\alpha) \qquad  
  (a = (a_\alpha)_\alpha \in c_{00}(\widehat\g)).
\end{displaymath}
We denote again $\|a\|_2 = \hat h(a^*a)^{1/2}$ the norm on $c_{00}(\hat\g)$
associated with this weight. By restriction and tensor product one obtains as
well norms, still denoted $\|\cdot\|_2$, on $B(H_{\alpha})$ and
$B(H_{\beta}\otimes H_{\gamma})$, associated to the inner products
$\langle a_1,a_2 \rangle= d_\alpha\mathrm{Tr} (Q_{\alpha}a_1^*a_2) $ for all
$a_1,a_2\in B(H_{\alpha})$ and
$ \langle x_1,x_2\rangle=d_\beta d_\gamma\mathrm{Tr}((Q_{\beta}\otimes
Q_{\gamma})x_1^*x_2) $
for all $x_1,x_2\in B(H_{\beta}\otimes H_{\gamma})$. Note that the collection of
matrices $Q_\alpha$ defines an algebraic (in general unbounded) multiplier
$Q = (Q_\alpha)_\alpha$ of $c_{00}(\hat\g)$, the {\em modular element}.
   
\bigskip
    
The analogue of the classical Fourier transform is the linear map
$\mathcal{F} : c_{00}(\widehat \g) \to C(\g)$ given by
$\mathcal{F}(a) = (\hat h\otimes \id)(V(a\otimes 1))$. Explicitly, we have
\begin{displaymath}
  \mathcal{F}(a)= \sum_{\alpha\in \mathrm{Irr}(\g)} \sum_{i,j=1}^{n_{\alpha}}
  d_{\alpha}(a_{\alpha}Q_{\alpha})_{j,i}u^{\alpha}_{i,j} \in C_r(\g).
\end{displaymath}
The Haar state $h$ on $\g$ and the left Haar weight $\hat h$ on $\widehat\g$ are
related through the Plancherel Theorem, which asserts that for any
$a = (a_\alpha)_{\alpha\in \mathrm{Irr}(\g)} \in c_{00}(\widehat \g)$, we have
$\hat h(a^*a) = h (\mathcal{F}(a)^*\mathcal{F}(a))$.

Let us note the following algebraic properties of the Fourier transform. Recall
that for $f \in \mathcal{O}(\g)$, $\varphi\in \mathcal{O}(\g)^*$ we denote
$f*\varphi = (\varphi\otimes\id)\Delta(f)$ and
$\varphi*f = (\id\otimes\varphi)\Delta(f)$. Then we have, for
$a\in c_{00}(\widehat\g)$, $\varphi \in \mathcal{O}(\g)^*$:
\begin{displaymath}
  \varphi*\mathcal{F}(a) = \mathcal{F}(ba) 
  \qquad\text{and}\qquad
  \mathcal{F}(a)*\varphi = \mathcal{F}(ab^Q) 
\end{displaymath}
where $b = (\id\otimes\varphi)(V)$ is an algebraic multiplier of
$c_{00}(\widehat\g)$ and $b^Q = QbQ^{-1}$. On the other hand for $a$,
$b \in c_{00}(\widehat\g)$ we have
$\mathcal{F}(a)\mathcal{F}(b) = \mathcal{F}(a\star b)$ where $a\star b$ is the
unique element of $c_{00}(\widehat\g)$ such that
$(\hat h\otimes\hat h) (\widehat{\Delta}(c) (a\otimes b)) = \hat h(c (a\star b))$ for all
$c\in c_{00}(\widehat\g)$.  The map $a \otimes b \mapsto a\star b$ defined above
is referred to as the {\it convolution product} on $c_{00}(\widehat \g)$.

We say that $\widehat\G$ is finitely generated when $\g$ is a compact matrix
quantum group. Having fixed a generating subset in $\mathrm{Irr}(\g)$, we put
$p_n = \sum_{|\alpha|=n} p_\alpha \in c_{00}(\hat\g)$. This is compatible with
the notation $p_n \in B(L_2(\g))$ introduced previously, in the sense that we
have $\mathcal{F}(p_na)\xi_0 = p_n \mathcal{F}(a)\xi_0$ for any $n\in\n_0$ and
$a\in c_{00}(\hat\g)$.

\subsection{The free orthogonal quantum groups}
    
We now come to the main objects of study in this paper.  Let $N\in \n$,
$N\geq 2$ and $F\in GL_N(\Comp)$ such that $F\overline{F}=\pm 1$.  The {\it free
  orthogonal quantum group} is the compact quantum group
$O^+_F=(C(O_F^+), \Delta)$, where
\begin{enumerate}
\item $C(O_F^+)$ is the universal unital $C^*$-algebra generated by $N^2$
  elements $u_{i,j}$, $1\leq i,j\leq N$, satisfying the relations making $u$
  unitary and $u=(F\otimes 1)u^c(F^{-1}\otimes 1)$, where
  $u=\left (u_{i,j}\right )_{1\leq i,j\leq N}\in M_N(\Comp)\otimes C(O_F^+)$ and
  $u^c=\left (u_{i,j}^*\right )_{1\leq i,j\leq N}$.
\item $\Delta:C(O_F^+)\rightarrow C(O_F^+)\otimes C(O_F^+)$ is the unital
  $*$-homomorphism determined by
  $\Delta(u_{i,j}) = \sum_{k=1}^N u_{i,k}\otimes u_{k,j}$.
\end{enumerate}

The compact quantum group $O_F^+$ is a compact matrix quantum group and we
choose the {\em fundamental representation}
$u= (u_{i,j})_{i,j} \in B(\C^N)\otimes C(O_F^+)$, coming from the canonical
generators of $C(O_F^+)$, as the (unique) generating representation. Then it is
known from \cite{Ba96} that for each $k\in\n_0$ there is a unique irreducible
representation (up to equivalence) of length $k$, which is equivalent to its
conjugate. We denote this class $k$, yielding an identification of
$\mathrm{Irr}(O_F^+)$ with $\n_0$. We have in particular $u^0 = 1_{C(\g)}$ (the
{\it trivial representation}), and
$u^1 = u = (u_{i,j}) \in B(H_1) \otimes C(\g)$ with $H_1 = \Comp^N$.

One can check that $Q_1 = F^{\text{tr}}\bar F$, so that $d_1 = \Tr(F^*F)$. There exists a
unique $q \in (0,1]$ such that $d_1 = q + q^{-1}$ and we denote also
$N_q = d_1 = q + q^{-1}$. On the other hand one can see that
$\|Q_k\| = \|Q_1\|^k = \|F\|^{2k}$ for all $k\in\n$, and that $O_F^+$ is of Kac
type iff $F$ is unitary. This is typically the case of $F = I_N$ and we denote
in this case $O_N^+ := O_{I_N}^+$.

It is moreover known that $u^m\tp u^n$ is unitarily equivalent to
$u^{|m-n|}\oplus u^{|m-n+2|}\oplus \cdots \oplus u^{m+n}$. We denote by
$P_l=P^{m,n}_l$ the orthogonal projection from $H_m \otimes H_n$ onto $H_l$ for
any one of $l = |m-n|,|m-n|+2,\cdots,m+n$. We have in particular $n_0=1,n_1=N$,
$n_1 n_{k+1}=n_{k+2}+n_k$ for all $k\in \n_0$ and $d_0=1$, $d_1=N_q:=\Tr(F^*F)$,
$d_1d_{k+1}=d_{k+2}+d_k$ for all $k\in \n_0$. Finally, it was also shown by
Banica \cite{Ba96} that the {\em fundamental character} $\chi_1 = \sum_{i=1}^N u_{ii}$ is a
semicircular element (on $[-2,2]$) with respect to the Haar state.

\subsection{Property RD and its generalizations}
\label{subsec_intro_RD}
    
In the case when $\widehat\g$ is a classical discrete group $\Gamma$, the
Property of Rapid Decay amounts to controlling the norm of
$C_r(\g) = C^*_r(\Gamma)$ from above by the $2$-norm. More precisely a discrete
group $\Gamma$ has Property RD if there exists a polynomial $P$ such that
\begin{equation}
  \label{eq:class_RD}
  \|x\|_{C^*_r(\Gamma)} \leq P(k) \|x\|_2
\end{equation}
for all $k \in \n_0$ and all $x \in C^*_r(\Gamma)$ supported on elements of
length $k$ in $\Gamma$, with respect to some fixed length (for instance a word
length if $\Gamma$ is finitely generated). Note that the reverse inequality
$\|x\|_2 \leq \|x\|_{C^*_r(\Gamma)}$ is always true.
    
A quantum generalization of Property RD was introduced in \cite{Ve07} by means
of the same inequality~\eqref{eq:class_RD}, with appropriate notions of length
and support as introduced above. It was shown in the same article that Property RD holds for the
dual of $O_N^+$ but fails for the dual of any compact quantum group $\g$ which
is not of Kac type. Later, a modification of the quantum definition was proposed
in \cite{BhVoZa15} so as to accommodate non-Kac examples such as $SU_q(2)$, and
more generally quantum groups $\g$ with (classical) polynomial growth. This
modification is obtained by replacing the $2$-norm on the right-hand side of
\eqref{eq:class_RD} by a still ``easily computable'' twisted $2$-norm.
    
\bigskip
    
In this setting, ``easily computable'' means a norm of the form
$\|f\|_\varphi = \|\varphi\star f\|_2$ or $\|f\star\varphi\|_2$ for
$f \in \mathcal{O}(\g)$, with $\varphi \in \mathcal{O}(\g)^*$ fixed. Using the
fact that the Fourier transform is isometric, this can also be written
$\|\mathcal{F}(a)\|_\varphi = \|Da\|_2$ or $\|aD^Q\|_2$ for $a \in c_{00}(\hat\g)$, where
$D = (\id\otimes\varphi)(V)$, and these norms can indeed be computed by
multiplying matrices and summing their traces. In this picture the twisted
Property RD takes the form $\|\mathcal{F}(a)\|_{C_r(\g)} \leq P(k) \|Da\|_2$ or
$\| \mathcal{F}(a) \|_{C_r(\g)} \leq P(k) \|aD'\|_2$ if $\mathcal{F}(a)$ is of length $k$, for some
fixed algebraic multiplier $D$ or $D'$ of $c_{00}(\widehat\g)$. Observe that by
polar decomposition one can assume $D>0$ (resp. $D'\sqrt Q>0$) without changing
the associated twisted norm.

Of course one could always achieve such inequalities by taking a central
multiplier $D =(b_\alpha I_{H_{\alpha}})_\alpha$ with weights $b_\alpha$ growing
sufficiently rapidly (see for example \cite{VaVe07} and the discussion at the
beginning of Section~\ref{sec_weak_RD}). However, for some applications (e.g.,
to the metric approximation property \cite{Br12}, and to non-commutative
geometry \cite{BhVoZa15}), it is desirable to use ``natural'' or ``optimal''
elements $D$, $D'$.

We note that the authors of \cite{BhVoZa15} choose the twisted $2$-norm in such a way that
$\left \{\sqrt{n_{\alpha}}u^{\alpha}_{i,j}\right\}_{\substack{1 \le i,j\leq n_{\alpha} \\
  \alpha \in \text{Irr}(\g)}}$ forms an orthonormal basis,
as it is in the case of Kac type compact quantum groups. An easy inspection with
our conventions shows that the only twisted norm with this property is
$\|\mathcal{F}(a)\|_{\varphi} := \|a\sqrt C\|_2$, where
\begin{align}\label{C}
  C = \Big(\frac{d_\alpha}{n_\alpha}Q_\alpha\Big)_\alpha
\end{align} 
is the canonical element used in \cite{BhVoZa15} to define their twisted
$2$-norms. In the following definition we fix a multiplier
$D=(D_{\alpha})_{\alpha\in \mathrm{Irr}(\g)}$ of $c_{00}(\widehat\g)$, we
consider the associated twisted norms $\|a\|_{2,D} := \|aD\|_2$ for
$a \in c_{00}(\widehat\g)$, and we put
\begin{displaymath}
  \|f\|_{2,D} := \|\mathcal{F}^{-1}(f)\|_{2,D} = \|\mathcal{F}^{-1}(f)D\|_2
\end{displaymath}
for $f \in \mathcal{O}(\g)$. Observe that $D$ is uniquely determined by
$\|\cdot\|_{2,D}$ if we assume $D\sqrt Q \geq 0$.
    
\begin{definitionn}
  Let $\g$ be a compact matrix quantum group with a fixed family of generating
  irreducible representations and $D$ a multiplier of $c_{00}(\widehat\g)$.  We
  say that $\widehat{\g}$ has Property $RD_D$ if there exists a polynomial
  $P \in \R_+[X]$ such that for all $k\in\n_0$ and $f\in \mathcal{O}(\g)$ of
  length $k$, we have
  \[\|f\|_{C_r(\G)} \leq P(k) \|f\|_{2,D}. \]
\end{definitionn}
    
The property above can also be written
$\|\mathcal{F}(a)\|_{C_r(\g)} \leq P(k) \|a\|_{2,D}$ for all $k \in \n_0$ and
$a\in p_k c_{00}(\widehat\g)$.  Explicitly, Property RD$_D$ asks that
\begin{equation}
  \norm{\sum_{|\alpha|=k}
    \sum_{i,j=1}^{n_{\alpha}}d_{\alpha}(a_\alpha Q_{\alpha})_{j,i}u^\alpha_{i,j}}_{C_r(\g)}^2
  \leq P(k)^2 \sum_{|\alpha|=k} d_{\alpha} 
  \Tr(D_{\alpha}Q_{\alpha}D_{\alpha}^* a_{\alpha}^*a_{\alpha}).
\end{equation}
The Property RD considered in \cite{BhVoZa15} corresponds to the case
$D = \sqrt C$, which satisfies $D\sqrt Q \geq 0$ since $C$ commutes with $Q$. If $\g$ is of Kac type, the Property $RD_{\sqrt{C}}$ coincides with the property $RD$ in \cite{Ve07}. In particular, if $\widehat{\g}$ is a discrete group $\Gamma$, then Property $RD_{\sqrt{C}}$ is exactly same with the property RD of $\Gamma$.
    
\bigskip
    
We now restate \cite[Lemma~4.6]{Ve07} in a slightly more general form. Note that
in the case $\g = O_F^+$ equipped with the canonical generating representation,
there is only one irreducible representation $\alpha = k \in \mathrm{Irr}(\g)$
for each given length $k$, and the inclusions $u^l \subset u^k\otimes u^n$ are
multiplicity-free.
    
\begin{lemma}
  Let $\g$ be a compact matrix quantum group with a fixed family of generating
  irreducible representations. For $k$, $n \in \n_0$ and
  $\gamma \in \mathrm{Irr}(\g)$ we denote
  \begin{equation}
    \nu_{k,n}^\gamma = \sum_{|\alpha|=k,|\beta|=n} \frac{d_\alpha d_\beta}{d_\gamma} 
    ~ \mathrm{mult}(u^\gamma\subset u^\alpha \tp  u^\beta).
  \end{equation}
  Then the discrete quantum group $\widehat\g$ has Property $RD_D$ with respect
  to a multiplier $D$ {\bf iff} there exists a polynomial $P$ such that we have,
  for any $k$, $l$, $n \in \n_0$ and for every $a \in p_k c_0(\widehat\g)$,
  $b\in p_n c_0(\widehat\g)$:
  \begin{equation}\label{ineq1}
    \sum_{|\gamma|=l} \nu_{k,n}^\gamma
    \| \hat\Delta(p_\gamma) (a\otimes b) \hat\Delta(p_\gamma) \|^2_2 \leq 
    P(k)^2 \|aD\otimes b\|_2^2.
  \end{equation}
\end{lemma}
    
\begin{proof}
  The proof is a straightforward extension of the ideas in the proof of \cite[Lemma~4.6]{Ve07} using our
  notation. Let us recall the main ideas for the convenience of the reader.
    
  First of all, Property $RD_D$ is equivalent to the fact that
  $\|p_l f p_n\|_{C_r(\g)} \leq P(k) \|f\|_{2,D}$ for all $k$, $l$, $n\in\n_0$
  and $f \in \mathcal{O}(\g)$ of length $k$, see \cite[Proposition~3.5]{Ve07}
  and \cite[Proposition~3.4]{BhVoZa15}. Using the Fourier transform, this means
  that we require
  \begin{equation}\label{ineq:RD-equiv}
\|p_l \mathcal{F}(a)\mathcal{F}(b)\xi_0\|_2 \leq P(k) \|a\|_{2,D} \|b\|_2 = P(k) \|aD \otimes b\|_2,
\end{equation}
  for all $a$ of length $k$ and $b$ of length $n$. Moreover we have
  $\|p_l \mathcal{F}(a)\mathcal{F}(b)\xi_0\|_2 = \|p_l \mathcal{F}(a\star
  b)\xi_0\|_2 = \|\mathcal{F}(p_l(a\star b))\|_2 = \|p_l(a\star b)\|_2$
  --- indeed by definition of $p_l$ (in $c_0(\widehat\g)$ and $B(L_2(\g))$) and
  of $V$ we have
  $(1\otimes p_l)V(1\otimes\xi_0) = (p_l\otimes\id)V(1\otimes\xi_0)$. Then we
  can decompose into orthogonal components:
  $\|p_l(a\star b)\|_2^2 = \sum_{|\gamma|=l} \|p_\gamma(a\star b)\|_2^2$.
    
  Then by definition of the convolution product we can write, for any
  $c \in c_0(\widehat\g)$:
  \begin{align*}
    \hat h(c^* p_\gamma(a\star b)) &= (\hat h\otimes\hat h)
                                (\hat\Delta(c)^*\hat\Delta(p_\gamma)(a\otimes b)) \\ 
                              &= (\hat h\otimes\hat h)
                                ((p_k\otimes p_n)\hat\Delta(c)^*\hat\Delta(p_\gamma)(a\otimes b)
                                \hat\Delta(p_\gamma)). 
  \end{align*}
  Note that $p_\gamma$ is central in $c_0(\widehat\g)$ and that
  $\hat\Delta(p_\gamma)$ is $(\hat h\otimes\hat h)$-central. To
  obtain the expression of $\|p_\gamma(a\star b)\|_2^2$ which appears in the
  left-hand side of~\eqref{ineq1}, it remains to take the supremum over
  $c \in p_\gamma c_0(\widehat\g)$, with $\|c\|_2\leq 1$. We show below that we
  have in fact
  $\norm{(p_k\otimes p_n)\hat \Delta(c)}_2^2=\nu_{k,n}^{\gamma}\norm{c}_2^2$,
  which yields the correct expression
  $\|p_\gamma(a\star b)\|_2 = (\nu_{k,n}^\gamma)^{1/2}
  \|\hat\Delta(p_\gamma)(a\otimes b)\hat\Delta(p_\gamma)\|_2$
  so that~\eqref{ineq1} results from~\eqref{ineq:RD-equiv}.

  Indeed we have $\hat\Delta(Q) = Q\otimes Q$ in the multiplier algebra of
  $c_{00}(\widehat\g)\otimes c_{00}(\widehat\g)$, and on the matrix algebra
  $p_\gamma c_0(\widehat\g) = B(H_\gamma)$ the $*$-homomorphism
  $(p_\alpha\otimes p_\beta)\hat\Delta$ is an amplification with the same
  multiplicity as the inclusion $u^\gamma \subset u^\alpha \tp
  u^\beta$. Thus we have
\begin{equation}
\Tr\left ( (Q\otimes Q) (p_{\alpha}\otimes p_{\beta})\hat \Delta (d)\right )=\text{mult}(u^{\gamma}\subseteq u^{\alpha}\tp u^{\beta})\Tr(Q d)
\end{equation}
for any $d\in B(H_{\gamma})$.
As a
  result we can write
  \begin{align*}
    (\hat h\otimes \hat h)\left( (p_k\otimes p_n)\hat\Delta (p_\gamma d)
\right) &= 
                                                                    \sum_{|\alpha|=k, |\beta|=n} d_\alpha d_\beta 
                                                                    (\Tr\otimes\Tr)[(p_\alpha\otimes p_\beta)(Q\otimes Q)\hat\Delta (p_\gamma d)] \\
                                                                  &=  \sum_{|\alpha|=k, |\beta|=n} d_\alpha d_\beta 
                                                                    \mathrm{mult}(u^\gamma \subset u^\alpha \tp u^\beta) \Tr (Q p_\gamma d) 
                                                                    = \nu_{k,n}^\gamma \hat h(p_\gamma d).
  \end{align*}     
  Taking $d = c^*c$ we obtain
  $\norm{(p_k\otimes p_n)\hat \Delta(c)}_2^2=\nu_{k,n}^{\gamma}\norm{c}_2^2$ as
  claimed. 
\end{proof}

\section{On Property RD$_D$ for $\widehat{O_F^+}$} \label{RDfail}

In this section we turn our attention to the duals of the free orthogonal
quantum groups $O^+_F$, establishing some necessary conditions for Property
RD$_D$ to hold for a given multiplier $D$.  In this case Property $RD_D$ with
respect to a multiplier $D$ and a polynomial $P$ is characterized by the
following multiplicity-free version of \eqref{ineq1}:
\begin{equation}\label{eq:ORD1}
  \|\hat\Delta(p_l)(a\otimes b)\hat\Delta(p_l)\|_2 \leq \sqrt{\frac{d_l}{d_k d_n}} P(k) \|aD_k\otimes b\|_2,
\end{equation}
for all $k$, $l$, $n$ such that $u^l \subset u^k\otimes u^n$, $a\in B(H_k)$ and
$b\in B(H_n)$.
    
\bigskip
    
Here the $2$-norms are the ones coming from the weight $\hat h$, but one can use
as well the twisted Hilbert-Schmidt norms, e.g.  $\|a\|^2_\HS = \Tr(Q_ka^*a)$
for $a\in B(H_k)$, since these two norms only differ by a scalar factor
$\sqrt{d_k}$. Moreover, if we fix an isometric intertwiner (unique up to a
phase) $v = v_l^{k,n} \in \mathrm{Mor}(u^l,u^k\otimes u^n)$ we have
$\|\hat\Delta(p_l)(a\otimes b)\hat\Delta(p_l)\|_\HS = \| vv^* (a\otimes b) vv^*
\|_\HS = \| v^* (a\otimes b) v \|_\HS$
--- notice that the last norm is the twisted Hilbert-Schmidt norm on $B(H_l)$.
    
We can moreover give an explicit form to the intertwiners $v_l^{k,n}$ as
follows. For each $n\in\n_0$ the tensor power representation
$u^1 \tp \cdots\tp u^1$ contains a unique copy of $u^n$, we choose for $H_n$ the
corresponding subspace of $H_1^{\otimes n}$ and we denote
$P_n = p_1^{\otimes n}\hat\Delta^{n-1}(p_n) \in B(H_1^{\otimes n})$ the
corresponding orthogonal projection. We further fix an intertwiner (unique up to
a phase) $t_n \in \mathrm{Mor}(1,u^n\tp u^n)$ such that $\|t_n\| = \sqrt{d_n}$
and we consider the intertwiners
$A^{k,n}_l=(P_k\otimes P_n)(\mathrm{id}\otimes t_r\otimes \mathrm{id})P_l \in
\mathrm{Mor}(u^l, u^k\otimes u^n)$,
where $r = (k+n-l)/2$. One can then take
$v_l^{k,n} = \| A^{k,n}_l \|^{-1} A^{k,n}_l$.
    
The (operator) norm of $A^{k,n}_l$ can be explicitly computed, see for example
\cite[Lemma 4.8]{Ve07} or \cite[Equation (6) and Proposition 3.1]{BrCo18}.  This
norm happens to be controlled from below and above, up to factors depending only
the parameter $0 <q < 1$ given by $q+q^{-1} = \text{Tr}(F^*F)$, as follows:
\begin{align*}
  \frac{1}{d_r} \le \|A_l^{k,n}\|^{-2}&=\frac{1}{d_r} \prod_{s=1}^r \frac{(1-q^{2+2s})(1-q^{2l-2r+2s})(1-q^{2m-2r+2s})}{(1-q^{2k+2+2s})(1-q^{2s})^2} \\
                                      & \le \frac{1}{[r+1]_q} \Big(\prod_{s=1}^r \frac{1}{1-q^{2s}}\Big)^3 \le \frac{1}{d_r} \Big(\prod_{s=1}^\infty \frac{1}{1-q^{2s}}\Big)^3, 
\end{align*}
where $l = k+n-2r$.  If we put \begin{align} \label{eq:c} 1 < C(q) =
  \frac{1}{(1-q^2)}\Big(\prod_{s=1}^\infty \frac{1}{1-q^{2s}} \Big)^3,
                               \end{align}
                               and use the inequality
                               \[
                               (1-q^2)^3 \le \frac{d_kd_n}{d_ld_r^2} \le
                               (1-q^2)^{-2}
                               \]
                               which follows from the dimension formula
                               $d_n = \frac{q^{n+1} - q^{-n-1}}{q-q^{-1}}$, we
                               get
                               \begin{align} \label{norm}(1-q^2)^{3/2}\Big(\frac{d_l}{d_kd_n}\Big)^{1/2}
                                 \le \|A_k^{l,m}\|^{-2} \le
                                 C(q)\Big(\frac{d_l}{d_kd_n}\Big)^{1/2}.
                               \end{align}

                               Inequality \eqref{norm} shows that
                               $\|A_l^{k,n}\|^{-2}$ compensates exactly for the
                               analogous factor of the right-hand side of the
                               $RD$ inequality \eqref{eq:ORD1} which can thus be
                               rewritten in the equivalent formulation
                               \begin{equation} \label{eq:ORD2} \| (A^{k,n}_l)^*
                                 (a\otimes b) A^{k,n}_l \|_\HS \le P(k)
                                 \|aD_k\otimes b\|_\HS.
                               \end{equation}
                               In \eqref{eq:ORD2}, it is important to use the
                               twisted Hilbert-Schmidt norms since the matrix
                               spaces are no longer the same on both sides.

\begin{remark}
  The universal constant $C(q)$ defined in \eqref{eq:c} will make several
  appearances in the remainder of the paper.
\end{remark}

\bigskip
        
Note that we have $u^l \subset u^k\tp u^n$ {\bf iff}
$l\in\{|n-k|,|n-k+2|,\ldots, n+k\}$. One can obtain necessary conditions for
Property $RD_D$ by fixing the value of $l$. More specifically we say that the
dual of $O_F^+$ satisfies Property $RD_D^0$ (resp. $RD_D^{\mathrm{max}}$) for
some polynomial $P$ if the above inequality is satisfied for all $k=n \in\n_0$
and $l=0$ (resp. for all $k$, $n\in\n_0$ and for $l=k+n$).
    
In the case of $RD_D^0$ we have simply
$A^{k,n}_l = A^{n,n}_0 = t_n : \Comp \to H_n\otimes H_n$ so that Property
$RD_D^0$, with respect to the polynomial $P$, is equivalent to the fact that
\begin{equation}
  | t_n^* (a\otimes b) t_n | \le P(n) \|aD_n\otimes b\|_\HS.
\end{equation}
for all $n\in\n_0$ and $a$, $b\in B(H_n)$. Note that $t_n$ can be written
uniquely as $t_n(1) = \sum_i e_i\otimes j_n(e_i)$, where the anti-linear map
$j_n : H_n \to H_n$ does not depend on the chosen orthonormal basis $(e_i)_i$ of
$H_n$, and recall that we have $Q_n = j_n^* j_n$ and $j_n^2 = \pm\id$. Then one
can compute (recalling that $(\zeta | j\xi) = (\xi | j^*\zeta)$ for an
anti-linear map $j$):
\begin{equation}\label{eq:quadratic}
  t_n^* (a\otimes b) t_n = \Tr(j_n^*b^*j_n a)
\end{equation}
Summing everything up, one can reformulate Property $RD_D^0$ as follows:
       
\begin{definitionn}
  We say that $\widehat{O_F^+}$ has $RD_D^0$ if there exists a polynomial $P$
  such that
  \begin{equation}\label{ineq5}
    \left |t_n^* (a\otimes b)t_n \right |= \Tr(j_n^* b^* j_n a) 
    \leq P(n)\norm{aD_n\otimes b}_\HS
  \end{equation}
  for all $a,b\in B(H_n)$ and $n\in \n_0$.
\end{definitionn}
    
It turns out that Property $RD_D^0$ for the dual of $O_F^+$ can be explicitly
characterized in terms of the matrices $D_n$ (and $Q_n$), as follows.

\begin{proposition}\label{prop1}
  The discrete quantum group $\widehat{O_F^+}$ has $RD_D^0$ with respect to a
  polynomial $P$ if and only if
  \begin{equation}
    \norm{Q_n^{-1/2} D_n^{-1} Q_n}\norm{Q_n^{1/2}}\leq P(n)~for~all~n\in \n.
  \end{equation}    
\end{proposition}
    
\begin{proof}
  Note that (\ref{ineq5}) can be written as
  \begin{equation*}
    |\Tr(b^* a) |^2 \leq P(n)^2 \norm{aD_n \otimes j_n^*bj_n}_\HS^2=P(n)^2\Tr(D_n^*Q_n D_na^*a)\Tr(Q_nj_n^*b^*j_nj_n^*bj_n).
  \end{equation*}
  We have moreover $D_n^*Q_nD_n = \tilde D_n^2$, where
  $\tilde D_n = D_n\sqrt Q_n$ is positive, and
  $\Tr(Q_nj_n^*b^*j_nj_n^*bj_n) = \Tr (j_nj_n^*j_nj_n^*b^*j_nj_n^*b) = \Tr
  (Q_n^{-1}b^*Q_n^{-1}bQ_n^{-1})$.
  Now we note that, by the Cauchy-Schwarz inequality (for the untwisted
  Hilbert-Schmidt scalar product), the maximum of
  $|\Tr(b^* a) |^2 / \Tr(\tilde D_n^2a^*a)$ equals $\Tr(\tilde D_n^{-2}b^*b)$,
  attained at $a = b \tilde D_n^{-2}$, so that $RD_D^0$ is equivalent to
  \begin{align*}
    \Tr(\tilde D_n^{-2}b^*b) \leq P(n)^2 \Tr (Q_n^{-1}b^*Q_n^{-1}bQ_n^{-1})
  \end{align*}
  Replacing $b$ with $Q_n^{\frac{1}{2}}b Q_n$, the above can be written as
  \begin{equation*}
    \Tr(Q_n\tilde D_n^{-2} Q_nb^*Q_n b)\leq P(n)^2\Tr(b^*b).
  \end{equation*}
  We note that, for positive matrices $M,N\in B(H_n)_+$,
  \[\max_{b\neq 0}\frac{\Tr(Mb^*Nb)}{\Tr(b^*b)}=\left (
    \max_{b\neq 0}\frac{\norm{N^{\frac{1}{2}}b M^{\frac{1}{2}}}_{HS}}{\norm{b}
      _{HS}} \right )^2\]
  equals $\|M\| \|N\|$, attained at $b=\xi \eta^*\in B(H_n)$, where $\xi,\eta$
  are unit vectors chosen to satisfy
  $\norm{N^{\frac{1}{2}}\xi}=\norm{N^{\frac{1}{2}}}$ and
  $ \norm{M^{\frac{1}{2}}\eta}=\norm{M^{\frac{1}{2}}}$. Therefore, $RD_D^0$ is
  equivalent to
  \begin{equation*}
    \|Q_n\tilde D_n^{-2}Q_n\| \|Q_n\| = \| \tilde D_n^{-1}Q_n\|^2 \|Q_n^{1/2}\|^2\leq P(n)^2~\mathrm{for~all~}n\in \n.
  \end{equation*}
  which gives us the desired conclusion.
\end{proof}
    
    \begin{remark}
      The same techniques apply if one tries to twist the $2$-norm from the
      other side, i.e. if one considers inequalities of the form
      $|t_n^*(a\otimes b)t_n| \leq P(n) \| D_na \otimes b\|_\HS$. Then one
      arrives at the condition $\|D_n^{-1} Q_n^{1/2}\| \|Q_n^{1/2}\| \leq P(n)$,
      which is equivalent to the condition of Proposition~\ref{prop1} if $D$ and
      $Q$ commute.
    \end{remark}
    
    
    Recall that \cite{BhVoZa15} take $D_k^2 = C_k = \frac{d_k} {n_k} Q_k$ to
    verify $RD_D$ for $SU_q(2)$ (more generally, the Drinfeld-Jimbo
    $q$-deformations $G_q$). It seems reasonable to consider a continuous family
    of variants of this multiplier by taking
    $D_k=(\frac{d_k}{n_k})^{|s|/2}Q_k^{s/2}$ with $s\in \mathbb{R}$. Recalling
    that $\|Q_k\| = \|Q_k^{-1}\| = \|Q_1\|^k$, we see that in that case
    Proposition \ref{prop1} reads
    \[\left ( \frac{n_k}{d_k} \right )^{|s|}\norm{Q_1}^{(|1-s|+1)k}\leq
    P(k)^2~\mathrm{for~all~}k\geq 0.\]
    
    However, the following theorem shows that this inequality is not satisfied
    for any non-Kac $O_F^+$ as soon as $N\geq 3$.
    
    \begin{theorem} \label{thm_RD_fail} Let $N\geq 3$, $s\in \mathbb{R}$ and
      consider the multiplier $D(s) = (D_k)_{k \in {\N_0}}$, with
      $D_k=\left (d_k/n_k\right )^{|s|/2}Q_k^{s/2}$.  Then $ \widehat{O_F^+}$
      has Property $RD_{D(s)}^0$ if and only if $O^+_F$ is of Kac type. In
      particular, all non-Kac $O_F^+$ do not have property $RD_{D(s)}$.
    \end{theorem}
    
\begin{proof}
  We only need to consider the case where $O^+_F$ is not of Kac type. I.e., we
  assume $Q_1 \ne I$.  Since
  $\left ( \frac{n_k}{d_k} \right )^{| s|}\norm{Q_1}^{(|1- s|+1)k}\geq \left (
    \frac{n_k}{d_k}\norm{Q_1}^k \right )^{|s|}$
  for all $k\in \n$, it suffices to show that $\frac{n_k}{d_k}\norm{Q_1}^k$ has
  exponential growth, i.e.
  \[\liminf_{k\rightarrow \infty} \left (\frac{n_k\norm{Q_1}^k}{d_k}\right
  )^{\frac{1}{k}}>1.\]
    
  First of all, we have $\lim_{k\rightarrow \infty}d_k^{\frac{1}{k}}=f(d_1)$ and
  $\lim_{k\rightarrow \infty}n_k^{\frac{1}{k}} =f(n_1)$ where
  $f(t)=\frac{t+\sqrt{t^2-4}}{2}$. Let us denote by
  $\lambda_1\leq \cdots\leq \lambda_{N}$ the eigenvalues of $Q_1$. Then, since
  the spectrum of $Q_1$ is symmetric under inversion, we have $\lambda_1<1$ and
  $\lambda_2\leq 1$.
    
  Then, in the expansion of
  $(\lambda_1+\cdots+\lambda_{N})^2=\sum_{i,j=1}^N \lambda_i \lambda_j$, we have
  $\lambda_1^2,\lambda_1\lambda_2,\lambda_2\lambda_1<1$, $\lambda_2^2\leq 1$ and
  the other terms are smaller than $ \lambda_N^2$.  From this observation we
  obtain
  \[d_1^2< 4+(N^2-4)\lambda_N^2,\]
  which together with the obvious estimate $d_1< N\lambda_N$ yields
  \begin{align*}
    f(d_1)&=\frac{d_1+\sqrt{d_1^2-4}}{2}< \frac{N\lambda_N+\sqrt{(N^2-4)\lambda_N^2}}{2}=f(n_1)\norm{Q_1}.
  \end{align*}
  Hence, we have
  $\liminf_{k\rightarrow \infty} \left (\frac{n_k\norm{Q_k}}{d_k}\right
  )^{\frac{1}{k}}=\frac{f(n_1)\norm{Q_1}}{f(d_1)}>1$.
\end{proof}
    
\begin{remark}
  On the other hand, one might try to consider the "opposite" case of
  $RD_D^{\max}$, which is satisfied iff we have
  $\|P_{k+n}(a\otimes b)P_{k+n}\|_\HS \leq P(k) \|aD_k\otimes b\|_\HS$ for all
  $k,n\in \n_0$ and $a\in B(H_k)$, $b\in B(H_n)$. In fact, when $D$ commutes
  with $Q$ it turns out that $RD_D^{\max}$ is a consequence of $RD_D^0$, thank
  to Proposition \ref{prop1}.
  
  Indeed we always have
  $\| P_{k+n}(a\otimes b)P_{k+n}\|_\HS \leq \|a\otimes b\|_\HS$ (using the fact
  that $P_{k+n}$ commutes with $Q_k\otimes Q_n$), so that
  $\|a\|_\HS \leq P(k) \|a D_k\|_\HS$ implies $RD_C^{\max}$. Performing the same
  analysis as in the proof of Proposition \ref{prop1} we see that this stronger
  condition is equivalent to $\|D_k^{-1}\| \leq P(k)$ for all $k$. On the other
  hand when $D$ commutes with $Q$ we can write
  $\|D_k^{-1}\| = \|Q_k^{-1/2} D_k^{-1} Q_k Q_k^{-1/2}\| \leq \|Q_k^{-1/2}
  D_k^{-1} Q_k\| \|Q_k^{-1/2}\|$,
  which makes the connection with the characterization of $RD_D^0$ given at
  Proposition \ref{prop1} since $ \|Q_k^{-1/2}\| = \|Q_k^{1/2}\|$.
\end{remark}
 
\begin{remark}
  The analysis of the above two subcases of Property $RD_D$ leads us to ask
  whether Property $RD_D$ is equivalent to $RD_D^0$ for $\widehat{O^+_F}$?
  I.e., Is Property $RD_D$ equivalent to the inequalities
  $\|Q_k^{-1/2} D_k^{-1} Q_k\| \|Q_k^{1/2}\| \leq P(k)$, at least when $D$ and
  $Q$ commute?
\end{remark}

\subsection{A Weaker Variant of Property RD}
\label{sec_weak_RD}

Despite the failure of $RD_{\sqrt C}$ for non-amenable, non-Kac type orthogonal
free quantum groups, one can prove a weaker RD inequality (corresponding to a
larger multiplier $D$) which holds for all orthogonal free quantum groups, and
also for all discrete quantum groups with polynomial growth. This inequality was
already stated (without proof) and used in \cite{VaVe07}, see Remark~7.6
therein. We provide below a slightly more precise statement and a proof.  In the
following section, we will see how this weakened property RD is applicable to
find {\it almost} sharp optimal time estimates for ultracontractivity of heat semigroups
on $O^+_F$.
    
\begin{proposition}\label{prop_expRD}
  Let $F\in GL_N(\Comp)$ be such that $F\bar F = \pm I_N$, $N\geq 2$. Then for
  any $k$, $l$, $n\in \n_0$ and $a\in B(H_k) \subset \ell_\infty(\hat O^+_F)$ we
  have
  \begin{align*}
    \|p_l\mathcal{F}(a)p_n\| \leq C(q) \|F\|^{2k} \|a\|_2 \quad\text{and}\quad   
    \|\mathcal{F}(a)\| \leq C(q) (k+1) \|F\|^{2k} \|a\|_2,
  \end{align*}
  where $C(q) >1$ is the constant defined by~\eqref{eq:c} for $0 < q < 1$ such
  that $\Tr(F^*F) = q+q^{-1}$.
\end{proposition}

\begin{proof}
  We follow quite closely the proof of \cite[Theorem~4.9]{Ve07}, taking into
  account the twisting of Hilbert-Schmidt norms.  Starting again
  from~\eqref{eq:ORD1} and taking into account~\eqref{norm} as in the beginning
  of Section~\ref{RDfail} the first inequality will follow if we prove
  \begin{equation}\label{eq:expRD_sufficient}
    \| (A^{k,n}_l)^* (a\otimes b) A^{k,n}_l \|_\HS \le \|F\|^{2k} \|a\otimes b\|_\HS
  \end{equation}
  for any $k$, $l$, $n\in \n_0$ such that $u^l\subset u^k\tp u^n$, $a\in B(H_k)$
  and $b\in B(H_n)$. Since $P_l$ is an orthogonal projection, the left-hand side
  admits
  $\|(\id\otimes t_r^*\otimes \id)(a\otimes b)(\id\otimes t_r\otimes \id)\|_\HS$
  as an evident upper bound, where $r=(k+n-l)/2$ and we are using the twisted
  Hilbert-Schmidt norm on $B(H_{k-r}\otimes H_{n-r})$. (Note that the projection
  $P_l$ commutes with the matrix $Q_k\otimes Q_n$ defining the twisting of the
  Hilbert-Schmidt norm.)

  We decompose $a = \sum_i a_i\otimes E_i$ and
  $b = \sum_i j_r^*E_i j_r\otimes b_i$, where $a_i \in B(H_{k-r})$,
  $b_i \in B(H_{n-r})$ and $(E_i)_i$ is the basis of matrix units in $B(H_r)$
  corresponding to an orthonormal basis of eigenvectors of $Q_r$ in $H_r$. With
  this choice we have in particular that $(E_i)_i$ and $(j_r^* E_i j_r)_i$ are
  orthogonal bases with respect to the twisted Hilbert-Schmidt scalar product on
  $B(H_r)$, and one can moreover compute, if $E_i = e_p e_q^*$ and $e_p$, $e_q$
  are eigenvectors of $Q_r$ with respect to eigenvalues $\lambda_p$,
  $\lambda_q$: $\|E_i\|_\HS^2 = \lambda_q$,
  $\|j_r^*E_i j_r\|_\HS^2 = \lambda_q^{-2}\lambda_p^{-1}$. In particular we note
  that
  \begin{equation}\label{eq:HS_matrix_units}
    \|E_i\|_\HS^{-2} =\lambda_q\lambda_p \|j_r^*E_i j_r\|_\HS^2 \leq \|Q_r\|^2\|j_r^*E_i j_r\|_\HS^2.
  \end{equation}

  According to~\eqref{eq:quadratic} we can then write
  \begin{align*}
    \|(\id\otimes t_r^*\otimes \id)(a\otimes b)(\id\otimes t_r\otimes \id)\|_\HS &= \|\textstyle\sum_{i,j} \Tr(E_j^* E_i) (a_i\otimes b_j)\|_\HS 
                                                                                   = \|\textstyle\sum_i a_i\otimes b_i\|_\HS.
  \end{align*}
  Now we apply the triangle inequality and Cauchy-Schwartz inequality:
  \begin{align*}
    \| (A^{k,n}_l)^* (a\otimes b) A^{k,n}_l \|_\HS^2 &\leq \left(\textstyle\sum_i \|a_i\otimes b_i\|_\HS\right)^2 = \left(\textstyle\sum_i \|a_i\|_\HS \|b_i\|_\HS\right)^2 \\
                                                     &\leq \textstyle \sum_i\|a_i\|_\HS^2 \|E_i\|_\HS^2 \sum_i\|b_i\|_\HS^2 \|E_i\|_\HS^{-2} 
                                                       = \|a\|_\HS^2 \sum_i\|b_i\|_\HS^2 \|E_i\|_\HS^{-2}.
  \end{align*}
  Finally we have
  $\sum_i\|b_i\|_\HS^2 \|E_i\|_\HS^{-2} \leq \|Q_r\|^2 \sum_i\|b_i\|_\HS^2
  \|j_r^*E_i j_r\|_\HS^{2} = \|Q_r\|^2 \|b\|_\HS^2$
  by~\eqref{eq:HS_matrix_units}. Since
  $\|Q_r\| = \|Q_1\|^r = \|F\|^{2r} \leq \|F\|^{2k}$ we have
  proved~\eqref{eq:expRD_sufficient}.

  The second inequality in the statement follows from the first one by a
  standard argument, see \cite[Proposition~3.5]{Ve07}, using the fact that for
  any $n$ the tensor product $u^k\tp u^n$ has at most $k+1$ irreducible
  subobjects.
\end{proof}
    
\begin{remark}
  The Property above can be interpreted as Property $RD_D$ with respect to the
  {\em central} multiplier
  $D = \sum_{k\in\n_0} \|F\|^{2k} p_k = \sum_{k\in\n_0} \|Q_k\| p_k$ and the
  (constant) polynomial $P = C(q)$. Note that the element $C$ in \cite{BhVoZa15}
  satisfies $\sqrt C \leq D$ and thanks to \cite[Proposition~4.2]{BhVoZa15} this
  implies that Property $RD_D$, for this element $D$, is also satisfied by all discrete quantum groups
  of polynomial growth, and still reduces to the usual Property $RD$ for (classical) discrete groups.

 %
%
\end{remark}

\section{Applications: Ultracontractivity and Hypercontractivity of the Heat
  Semigroup on $O_F^+$}
\label{sec_applic}

In this section of the paper, we are interested in studying hypercontractivity
and ultracontractivity properties of the heat semigroup $(T_t)_{t > 0}$ on the
free orthogonal quantum groups $O^+_F$.  This heat semigroup was
introduced and studied in \cite{CiFrKu14,FrHoLeUlZh17} in the Kac-type setting (i.e.,
$F = I_N$), but a standard argument using results from \cite{dCFY, Fr13} on
monoidal equivalences and transference properties of central multipliers allows
one to define an appropriate heat semigroup on all free orthogonal quantum
groups $O^+_F$'s.  The details of this are spelled out, for example, in
\cite[Section 6.1]{Ca18}.

Let $M$ be a von Neumann algebra equipped with a fixed faithful normal state
$\varphi$.  In the following, a {\it $\varphi$-Markov semigroup on $M$} will
mean a $\sigma$-weakly continuous semigroup $(T_t)_{t \ge 0}$ of normal unital
completely positive $\varphi$-preserving maps $T_t: M \to M$.  With a slight
abuse of notation, we will identify $M \subset L_2(M)$ as a dense subspace (via
the GNS map associated to $\varphi$) also denote by $T_t:L_2(M) \to L_2(M)$ the
canonical extension of $T_t$ to a contraction on the GNS space $L_2(M)$.  The
semigroup $(T_t)_{t\ge 0}$ is called {\it ultracontractive} if there exists some
$t_\infty \ge 0$ such that $T_t(L_2(M)) \subset M$ for all $t >t_\infty$. By the closed graph theorem, ultracontractivity is equivalent to that $T_t:L_2(M)\rightarrow M$ is bounded for all $t>t_{\infty}$. We
call $(T_t)_{t \ge 0}$ {\it hypercontractive} if for each $2 < p < \infty$,
there exists a $t_p > 0$ such that for all $t \ge t_p$, we have:
\[
\|T_t\|_{L_2(M) \to L_p(M)} \le 1.
\]
(In the above, we have used the standard fact \cite{JuXu03} that the
contractions $T_t$ admit canonical extensions to contractions
$T_t: L_p(M) \to L_p(M)$ on the associated non-commutative $L_p$-spaces for all
$p \in [1, \infty)$.  We omit the precise details regarding these extensions
here because in the following we only consider hypercontractivity in the tracial
setting.)  In the case of ultracontractive (resp. hypercontractive) semigroups $(T_t)_t$ the {\it
  optimal time} $t^o_\infty$ (resp. $t_p^o$) for ultracontractivity (resp. hypercontractivity) is given by
$t^o_\infty = \inf\{t_\infty\}$ (resp. $t_p^o = \inf\{t_p\}$).

Let us now consider the heat semigroup on $O^+_F$.

\subsection{The heat semigroup on $O^+_F$}
    
Fix $N \in \N$ and $F \in \text{GL}_N(\C)$ with $F \bar F = \pm 1$.  Let
$0 < q <1$ be such that $N_q:=\text{Tr}(F^*F) = q+q^{-1}$, and define
\begin{equation}
  \label{eq:def_lambda}
  \lambda(k) = \lambda_q(k) =\frac{U_k'(N_q)}{U_k(N_q)} \qquad (k \in \N_0),
\end{equation}
where $U_k$ is the $k$-th type-II Chebychev polynomial (defined by $U_0(x) = 1,$
$U_1(x) = x$, and $xU_k(x) = U_{k+1}(x)+ U_{k-1}(x)$). The {\it heat semigroup
  on $O^+_F$} \cite{FrHoLeUlZh17, Ca18} is the $h$-Markov semigroup
$(T_t)_{t \ge 0}$ on $L_\infty(O^+_F)$ given by
\[
T_t \left (u^k_{i,j}\right ) = e^{-t \lambda(k)} u^k_{i,j},
\]
for all $1\leq i,j\leq n_k$, and $k \in \N_0$.
    
Note that we have $\lambda(0) = 0$, $\lambda(1) = 1/N_q$ and moreover from the
estimates in \cite{FrHoLeUlZh17} we have
$\frac{k}{N_q} \le \lambda(k) \le \frac{k}{N_q-2}$ for all $k\in\n$.
    
    
\subsection{Ultracontractivity of the Heat Semigroup on $O^+_F$}

We first consider the ultracontractivity of the heat semigroups.  In the tracial
case, the ultracontractivity of the heat semigroup {\it for all time} (with
$t_\infty = 0$) is well known and follows from standard tracial property RD
estimates.  See \cite[Theorem 2.1]{FrHoLeUlZh17}.  In the case of general
$O^+_F$, we show below that ultracontractivity still holds, but generally not
for all time.
    
\begin{proposition}\label{prop_ultra_non_kac}
  The heat semigroup $(T_t)_{t \ge 0}$ is ultracontractive for every free
  orthogonal quantum group $O^+_F$.  Moreover, if $t_F$ is the optimal time for
  ultracontractivity of the heat semigroup of $O_F^+$ we have
  \begin{equation}
    2(N_q-2)\log \norm{F}\leq t_F\leq 2 N_q\log \norm{F}.
  \end{equation}
\end{proposition}
    
\begin{proof}
  First of all, let us suppose that $t>2 N_q \log \norm{F}$ and take
  $f \in L_2(O_F^+)$.  Then we can write $f = \sum_{k\geq 0} f_k$, with
  $f_k \in \mathrm{span}\left \{u^k_{i,j}:1\leq i,j\leq n_k\right\}$.  Using the
  exponential form of Property RD given by Proposition \ref{prop_expRD}, we then
  have
  \begin{align*}
    \|T_t (f)\|_\infty &\le \sum_{k\geq 0} e^{-\lambda(k)t}\|f_k \|_\infty \\
                       & \le \sum_{k\geq 0} e^{-\lambda(k)t}C(q)(k+1)\|F\|^{2k}\|f_k\|_2 \\
                       &\le C(q) \Big(\sum_{k\geq 0} e^{-2\lambda(k)t}(k+1)^2\|F\|^{4k}\Big)^{1/2}\|f\|_2.
  \end{align*}
  Hence the conclusion follows if
  \[e^{-2\lambda(k)t}\norm{F}^{4k}\leq e^{-Mk}\]
  for all $k\in \n_0$ by a universal constant $M>0$. Indeed, let
  $M= \frac{2}{N_q}t-4\log \norm{F}>0$. Then
  \[\inf_{k\in \n_0} \left \{ \frac{\lambda(k)}{k}t-2\log \norm{F}
  \right\}\geq\frac{1}{N_q}t-2\log \norm{F}=\frac{M}{2}>0,\]
  which completes the proof.
    
  To prove the stated lower bound, let us assume that
  $\norm{T_t(f)}_{L_{\infty}( O_F^+)}\leq K \norm{f}_{L_2(O_F^+)}$ for a
  universal constant $K>0$.  Then for any $k\geq 0$ and $1\leq i,j\leq n_k$ we
  have
  \begin{align*}
    e^{-t\lambda(k)}\norm{(u^k_{i,j})^*}_{L_2(O_F^+)}&\leq e^{-t\lambda(k)}\norm{(u^k_{i,j})^*}_{L_{\infty}(O_F^+)}\\
                                                     &=e^{-t\lambda(k)}\norm{u^k_{i,j} }_{L_{\infty}(O_F^+)}\leq K\norm{u^k_{i,j}}_{L_2(O_F^+)}.
  \end{align*}
  On the other hand, using~\eqref{eq:ortho_relations} it is easy to compute
  $\|u^k_{i,j}\|^2_{L_2(O_F^+)} = d_k^{-1} (Q_k^{-1})_{ii}$ and
  $\|(u^{k}_{i,j})^*\|^2_{L_2(O_F^+)} = d_k^{-1} (Q_k)_{jj}$.  Therefore, using a
  basis for $H_k$ in which $Q_k$ is diagonal we obtain
  \[\norm{F}^{2k}=\sqrt{\norm{Q_k}\norm{Q_k}}=\max_{i,j}\frac{\norm{(u^k_{i,j})^*}_{L_2(O_F^+)}}{\norm{u^k_{i,j}
    }_{L_2(O_F^+)}}\leq K\cdot e^{t\lambda(k)}\leq K\cdot
  e^{\frac{tk}{N_q-2}},\]
  which implies $t\geq 2(N_q-2)\log \norm{F} - \frac{(N_q-2)\log(K)}{k}$ for all
  $k$. Then, taking the limit $k\rightarrow \infty$ gives the desired
  conclusion.
\end{proof}

\begin{remark}
  A closer examination of the above proof actually shows that
  $T_t(L_2(O^+_F)) \subset C_r(O^+_F)$ for all $t > t_F$.  I.e., the heat
  semigroup on $O^+_F$ has some additional ``smoothing'' properties beyond what
  is guaranteed by ultracontractivity.
\end{remark}

\begin{remark}
  Of course, it is natural to wonder if hypercontractivity holds for the heat
  semigroups of {\it all} free orthogonal quantum groups $O^+_F$. Actually we
  can show that hypercontractivity is always obtained, although at this time we
  have no clue for optimal estimates for the time to contraction.
\end{remark}
   
\begin{proposition}\label{prop_hyper_non_kac}
  Let $2 < p < \infty$.  For sufficiently large $t$ (depending on $p$),
  $T_t:L_2(O_F^+) \to L_p(O_F^+)$ is a contraction.
\end{proposition}
    
\begin{proof}
  For any $f \in L_2(O_F^+)$, we have from \cite[Theorem 1]{RiXu16},
  \begin{align*}
    \|T_t (f)\|_p^2 &\le \|h(T_t(f))1\|_p^2 + (p-1)\|T_t(f )  - h (T_t (f))\|_p^2 \\
                    &\le|h( f ) |^2 +(p-1) \Big(\sum_{n \ge 1} e^{-\lambda(n)t}\|f_n\|_p \Big)^2\\
                    &\le|h(f )|^2 + (p-1)\Big(\sum_{n \ge 1} e^{-\lambda(n)t}\|f_n\|_\infty \Big)^2\\
                    & \le|h(f )|^2 +(p-1) \Big(\sum_{n \ge 1} e^{-\lambda(n)t}C(q)(n+1)\|F\|^{2n}\| f_n \|_2 \Big)^2\\
                    & \le |h(f )|^2 + (p-1)\Big(\sum_{n \ge 1} e^{-2\lambda(n)t}C(q)^2(n+1)^2\|F\|^{4n}\Big)\|f - h(f)1\|_2^2 \le \|f \|_2^2,
  \end{align*}
  for all $t$ large enough so that
  \begin{displaymath}\label{prelim-bound}
    \sum_{n \ge 1} (p-1)e^{-2\lambda(n)t}C(q)^2(n+1)^2\|F\|^{4n} \le 1.
  \end{displaymath}
\end{proof}

\subsection{Improved Hypercontractivity Results for $O_N^+$}

For the remainder of the paper we turn our attention to the Kac setting and
consider $O^+_N$.  Our aim is to revisit the hypercontractivity results of
\cite{FrHoLeUlZh17}, and obtain some improved estimates (from above and below)
on the optimal time to contraction for the heat semigroup $(T_t)_{t \ge 0}$.  In
the following we let $t_{N,p}$ be the optimal time for $L_2 \to L_p$
hypercontractivity of the heat semigroup on $O_N^+$.

We begin with a necessary lower bound for $t_{N,p}$.

\begin{lemma}\label{lem-necessary}
  For each $N \ge 2$ and $2 <p < \infty$, we have
  \[t_{N,p}\geq \frac{N} {2}\log(p-1).\]
\end{lemma}
\begin{proof}
  Let $\chi_1$ denote the character of the fundamental representation of
  $O^+_N$.  With $f_a= 1+a\chi_1\in L_2(O_N^+)$ and sufficiently small $a>0$, we have
  \begin{align*}
(1+a^2)^{\frac{p}{2}} =\norm{f_a}_{L_2(O_N^+)}^p&\geq \norm{T_{t_{N,p}}(f_a)}_{L_p(O_N^+)}^p\\
=&\frac{1}{2\pi}\int_{-2}^{2} (1+ae^{-\frac{t_{N,p}}{N}}x)^p \sqrt{4-x^2}dx\\
=&\frac{2}{\pi} \int_{-\frac{\pi}{2}}^{\frac{\pi}{2}}(1+2a e^{-\frac{t_{N,p}}{N}} \sin(\theta))^p \cos^2(\theta) d\theta .
  \end{align*}
Then, using Taylor expansion up to second order, we can obtain
\[\frac{p}{2}=\lim_{a\searrow 0}\frac{(1+a^2)^{\frac{p}{2}}-1}{a^2}\geq \lim_{a\searrow 0} \frac{\frac{2}{\pi} \int_{-\frac{\pi}{2}}^{\frac{\pi}{2}}(1+2a e^{-
      \frac{t_{N,p}}{N}} \sin(\theta))^p \cos^2(\theta) d\theta -1}{a^2}=\frac{p(p-1)e^{-\frac{2t_{N,p}}{N}}}{2}.\]
Equivalently, we have  $\displaystyle t_{N,p}\geq \frac{N} {2}\log(p-1)$.

\end{proof}

\subsubsection{Khintchine Inequalities for $L_p(O_N^+)$}
    
Our next goal is to establish upper bounds for the optimal time to contraction
$t_{N,p}$.  To do this, we follow along the lines of \cite[Lemma 7]{RiXu16},
establishing and then exploiting a certain non-commutative Khintchine type
inequality over $O_N^+$.  More precisely, we are interested in finding the
optimal constants $K_{m,p}>1$ such that for all $m \in \N$, $p > 2$, and
$f \in \mathcal{O}(\g)$ of length $m$ we have
\begin{displaymath} \label{Khintchine} \|f\|_{L_p(O_N^+)} \leq K_{m,p}
  \|f\|_{L_2(O_N^+)}.
\end{displaymath}

\begin{theorem} \label{thm:Khintchine} For $O_N^+$ we have the following
  estimates for $K_{m,p}$:
  \begin{enumerate}
  \item $\displaystyle K_{m,p} \le (C(q)^2(m+1))^{1- \frac{3}{p}}$
    $(4 < p \le \infty)$,
  \item $\displaystyle K_{m,p} \le (C(q)^2(m+1))^{\frac{1}{2}- \frac{1}{p}}$
    $(2 \le p \le 4)$.
  \end{enumerate}
\end{theorem}
    

\begin{proof}
  For any admissible triple $(l,m,n)\in \n_0^3$ let
  $v_l^{m,n} = \|A_l^{m,n}\|^{-1}A_l^{m,n} \in \Mor_{O^+_N}(H_l, H_m \otimes
  H_n)$
  be the isometric intertwiner considered in Section \ref{RDfail}.  If we repeat
  the usual RD-type calculations for $O_N^+$ (e.g. \cite[Section 4]{Ve07} or
  \cite[Section 5]{BrRu17}), one obtains the following general inequality for
  the (untwisted) Hilbert-Schmidt norms
  \begin{align*}
    \Big\| (v_l^{m,n})^* (y \otimes z) v_{l}^{m,n}\Big\|_{HS} \le  \|A_{l}^{m,n}\|^{-2}\|y\|_{HS}\|z\|_{HS} \le C(q) \Big(\frac{d_l}{d_md_n}\Big)^{1/2}\|y\|_{HS}\|z\|_{HS}.
  \end{align*}
  for any $y \in B(H_m)$, $z \in B(H_n)$.  Note that the second inequality above
  follows from \eqref{norm}.

  We now consider the case $p = \infty$.  In this case, we note that the above
  inequality is exactly the required estimate \eqref{eq:ORD1} for Property RD to
  hold: It says that $\|p_lfp_n\| \le C(q) \|f\|_2$ for each
  $f \in \text{span}\left \{u^m_{i,j}:1\leq i,j\leq n_m\right\}$.  This implies
  that $\|f\|_{C_r(O_N^+)} \le C(q)(m+1)\|f\|_2$ for all $m \in \N_0$ and all
  $f\in \text{span}\left \{u^m_{i,j}:1\leq i,j\leq n_m\right\}$.  I.e., we have
  $K_{m,\infty} \le C(q)(m+1) \le C(q)^2(m+1)$.

  Next, we consider $p=4$.  Now, we define an involution structure $\sharp$ on
  $ B(H_m)$ by $a^{\sharp}=J_m^{-1} \overline{a} J_m$ for all $a\in B(H_m)$,
  where $J_m$ is the unique anti-unitary satisfying
  $(u^m)^c=(J_m\otimes 1)u^m (J_m^{-1}\otimes 1)$. Then, for any
  $f= \sum_{i,j=1}^{n_m}a_{j,i}u^m_{i,j}\in C_r(O_N^+)$, we have
  \[f^*f=\sum_{s=0}^m \sum_{i,j=1}^{n_m} [(v^{m,m}_{2s})^*(a^{\sharp}\otimes
  a)v^{m,m}_{2s}]_{j,i}u^{2s}_{i,j}.\]
  
  Thus
  \begin{align*}
    \|f\|_{L_4(O_N^+)}^4  &= \|f^*f\|_{L_2(O_N^+)}^2= \sum_{s=0}^m \frac{1}{d_{2s}}\Big\|(v_{2s}^{m,m})^*(a^\sharp \otimes a) v_{2s}^{m,m}\Big\|_{HS}^2 \\
                          &\le \sum_{s=0}^m \frac{C(q)^2d_{2s}}{d_{2s} d_{m}^2}\|a^\sharp\|_{HS}^2\|a\|_{HS}^2 \\
                          &= C(q)^2\frac{\|a\|_{HS}^4}{d_m^2} (m+1) \\
                          &= C(q)^2(m+1)\|f\|_2^4.  
  \end{align*}
  Thus $K_{m,4} \le \sqrt{C(q)}(m+1)^{1/4}$.  The rest of the proof now follows
  from complex interpolation theorem and our estimates for
  $K_{m,\infty}, K_{m,4}$.
\end{proof}
    
\begin{remark}
  The above bound for $K_{m,4}$ is essentially optimal, since
  $\|\chi_m\|_4 =(m+1)^{1/4}$ (the $4$th moment of the $m$th type II Chebychev
  polynomial).
\end{remark}

\subsubsection{Applications to Improved Optimal Time Estimates.}

\begin{theorem}\label{thm1}
  Let $p\geq 4$ and $c_p=1+\displaystyle \frac{4}{\log(p-1)}$. Then we have
  \[t_{N,p}\leq \frac{c_pN}{2}\log(p-1)+\left(1-\frac{3}{p}\right)\cdot
  2N\log(C(q)). \]
\end{theorem}
\begin{proof}
  By the non-commutative martingale convexity inequality of \cite[Theorem
  1]{RiXu16} and our Theorem \ref{thm:Khintchine}, we have \small
  \begin{align*}
    \norm{T_t(f)}_{L_p(O_N^+)}^2&\leq h(f)^2+(p-1)\norm{\sum_{k\geq 1}e^{-t\lambda(k)}f_k}_{L_p(O_N^+)}^2\\
                                &\leq h(f)^2+(p-1) \left (\sum_{k\geq 1}e^{-t\lambda(k)}C(q)^{2(1-\frac{3}{p})} (1+k)^{1-\frac{3}{p}}\norm{f_k}_{L_2(O_N^+)} \right )^{2}\\
                                &\leq h(f)^2+ (p-1)\left (\sum_{k\geq 1}e^{-2t\lambda(k)}C(q)^{4(1-\frac{3}{p})} (1+k)^{2-\frac{6}{p}}\right )\left ( \sum_{k\geq 1}\norm{f_k}
                                  _{L_2(O_N^+)}
                                  ^2\right )
  \end{align*}
  \normalsize for any $f\in L_p(O_N^+)$ and $p\geq 4$.
    
  Note that, for any $c\geq 1$, the assumption
  $t\geq \frac{c N}{2}\log(p-1)+2N (1-\frac{3}{p})\log(C(q) )$ implies \small
  \begin{align*}
    t&\geq \frac{c N}{2}\log(p-1)+\frac{2N}{k}\left (1-\frac{3}{p}\right )\log(C(q)) \\ &\geq \frac{c k }{2\lambda(k)}\log(p-1)+\frac{2}{\lambda(k)}\cdot \left (1-\frac{3}{p}\right )
                                                                        \log(C(q)).
  \end{align*}
  \normalsize Here, the second inequality results from the estimate
  $\lambda(k)\geq \frac{k}{N}$, where the $\lambda(k)$ are the
  coefficients~\eqref{eq:def_lambda} used in the definition of the generalized
  heat semigroup. Thus we can write 
  $ e^{-2t\lambda(k)}C(q)^{4(1-\frac{3}{p})} \leq (p-1)^{-c k}$. Now, let us try
  to find $c\geq 1$ satisfying
  \[\phi(c):=\sum_{k\geq 1}(p-1)^{1-c k}(1+k)^{2-\frac{6}{p}}\leq 1.\]
    
  To do this, we will use the following estimation
  \begin{align*}
    \phi(c)&\leq \sum_{k\geq 1}(p-1)^{1-ck}(1+k)^2 \\
           &= (1-(p-1)^{-c})^{-3}(p-1)(4(p-1)^{-c}-3(p-1)^{-2c}+(p-1)^{-3c})=:\psi(c).
  \end{align*}
    
  By setting $t=\frac{1}{1-(p-1)^{-c}}$, the problem to find $c\geq 1$
  satisfying $\psi(c)\leq 1$ becomes equivalent to solve the following
  inequality
  \[2t^3-t^2-1\leq \frac{1}{p-1} \Leftrightarrow t^2(2t-1)\leq \frac{p}{p-1}.\]
    
  Now, our claim is that the above inequality holds at $t=1+\frac{a}{p-1}$ with
  $a=\frac{1}{24}$.  Indeed, since $(1+x)^3\leq 1+3x(1+x)^2$ for all $x>0$, we
  have
  \begin{align*}
    \left (1+\frac{a}{p-1}\right )^2\cdot \left (1+\frac{2 a}{p-1} \right )&\leq \left (1+\frac{2a}{p-1}\right )^3\\&
                                                                                                                      \leq 1+\frac{6a}{p-1}\cdot (1+\frac{2a}{p-1})^2 \leq 1+\frac{24a}{p-1}= \frac{p}{p-1}.
  \end{align*}
    
  Therefore, we can see that $\phi(c)\leq \psi(c)\leq 1$ at
  $c=1+\frac{\log(24-\frac{23}{p})}{\log(p-1)}$. Lastly, since
  $c_p=1+\frac{4}{\log(p-1)}\geq 1+\frac{\log(24)}{\log(p-1)}\geq c$ and $\phi$
  is decreasing, we have $\phi(c_p)\leq \phi(c) \leq 1$.
\end{proof}
    
Theorem \ref{thm1} sharpens \cite[Theorem 2.6]{FrHoLeUlZh17} in the case when
\begin{equation*}
  1+\frac{4}{\log(p-1)}\leq \frac{2\log(1+ \sqrt{3})}{\log(3)}\approx 1.8297,
\end{equation*}
i.e. when $p\geq 125.1085$ approximately. However, even for
$2\leq p\leq 125.1085$, we can obtain an improved time to contractivity:

\begin{theorem}\label{thm2}
  Let $c=\frac{9}{8}\log(2)+1\approx 1.7798$ and $p\geq 4$. Then
  \[\norm{T_t(f)}_{L_p(O_N^+)}\leq \norm{f}_{L_2(O_N^+)}\]
  for all $t\geq \frac{cN}{2}\log(p-1)+(1-\frac{3}{p})\cdot 2N\log(C(q))$.
\end{theorem}
\begin{proof}
  In the proof of Theorem \ref{thm1}, let
  $\phi_k(p)=(p-1)^{1-ck}(1+k)^{2-\frac{6}{p}}$. Then
  \[\phi_k'(p)=(p-1)^{-ck}(1+k)^{2-\frac{6}{p}}((1-ck)+\frac{6(p-1)}{p^2}\log(1+k)).\]
  Let us suppose that $1\leq c\leq 2$ and consider functions
  $f(p)=\frac{6(p-1)}{p^2}$ and $g(k)=\frac{ck-1}{\log(1+k)}$.  Then it is easy
  to check that $g'(k)\geq 0$ for all $k\geq 1$ and $f'(p)=6p^{-3}(2-p)$.
    
  Since $f(4)= \frac{9}{8}\leq g(1)=\frac{c-1}{\log (2)}$ for all
  $c\geq \frac{9}{8}\log(2)+1\approx 1.7798$, the function $ \phi_k$ is
  decreasing on $[4,\infty)$ for each $k\geq 1$. Therefore,
  $\phi= \sum_{k\geq 1}\phi_k$ is decreasing on $[4,\infty)$ and
  \begin{align*}
    \phi(4)&=3 \sum_{k\geq 1}3^{-ck}(1+k)^{\frac{1}{2}}\\
           &\leq 3\left (\sum_{k\geq 1}3^{-ck}(1+k)\right )^{\frac{1}{2}}\left (\sum_{k\geq 1}3^{-ck}\right )^{\frac{1}{2}}=3(1-3^{-c})^{-\frac{3}{2}}3^{-c}(2-3^{-c})^{\frac{1}{2}}\leq 1
  \end{align*}
  if and only if $3^{-c}\leq X_0$, where $X_0$ is the second largest solution
  of the equation $8X^3-15X^2-3X+1\geq 0$. Hence, $\phi(p)\leq 1$ for all
  $p\geq 4$ whenever $c\geq -\log_3(X_0)\approx 1.547326$.
\end{proof}
    
In \cite{RiXu16}, their $L_p-L_2$ Khintchine inequalities were used in
conjunction with a clever choice of conditional expectation onto the subalgebra
generated by a semicircular system to find the optimal time $t_{N,p}$ for heat
semigroups of free groups.  However, it is not clear what would be the right
choice of subalgebra to play the same game for the free orthogonal quantum
groups $O_N^+$. Nevertheless, our Khintchine inequalities (Theorem
\ref{thm:Khintchine}) enable us to get an almost optimal time to contraction
under the additional assumption that $h(f u_{i,j})=0$ for all $1\leq i,j\leq N$:

    
\begin{theorem}\label{thm3}
  Let $N\geq 3$ and $p\geq 4$. Then the following inequality
  \[\norm{T_{t}(f)}_{L_p(O_N^+)}\leq \norm{f}_{L_2(O_N^+)}\]
  holds
  \begin{enumerate}
  \item if $f \in L_2(O_N^+)$ satisfies $h(fu_{i,j})=0$ for all
    $1\leq i,j\leq N$ and
  \item if $t\geq \frac{N}{2}\log(p-1)+(1-\frac{3}{p})\cdot 2N\log(C(q))$.
  \end{enumerate}
\end{theorem}
    
\begin{proof}
  By repeating the proof of Theorem \ref{thm1}, since $c=1$ and $f_1=0$, the calculation can be distilled to show
\[\phi(p)=\sum_{k\geq 2} (p-1)^{1-k}(1+k)^{2-\frac{6}{p}}\leq 1\text{ for all }p\geq 4.\]

It is easy to check that $\phi_k(p)=(p-1)^{1-k}(1+k)^{2-\frac{6}{p}}$ is a decreasing function for any $k\geq 3$ and $\displaystyle \sup_{p\geq 4}\phi_2(p)=\sup_{p\geq 4}\left \{ (p-1)^{-1}3^{2-\frac{6}{p}}\right \}\approx 0.60348$. Thus, 
\begin{align*}
\phi(p)&\leq \sup_{p\geq 4}\phi_2(p)+\sum_{k\geq 3}3^{1-k}\sqrt{1+k}\\
&=\sup_{p\geq 4}\phi_2(p)+ 9 \sum_{k\geq 4}3^{-k}\sqrt{k}\\
&\leq \sup_{p\geq 4}\phi_2(p)+ 9 \left (\sum_{k=4}^6 3^{-k}(\sqrt{k}-k)+\sum_{k\geq 4}k3^{-k}\right )\\
&\approx 0.60348 + 0.38158<1.
\end{align*}

\end{proof}

Based on the above results, we are led to make the following conjecture on the asymptotic behavior of the optimal time-to-contraction for the heat semigroups:

\begin{conjecture}
  The optimal time to $L_p$-hypercontractivity for $O_N^+$ should be of the form
  \begin{equation}
    t_{N,p} = \frac{N}{2} \log(p-1)+\epsilon_N \quad\text{with}\quad 
    \lim_{N\rightarrow \infty}\epsilon_N=0.
  \end{equation}
\end{conjecture}

\begin{remark}
  The conjecture above is motivated by the following observations:
  \begin{enumerate}
  \item We have $t_{N,p}\geq \frac{N}{2}\log(p-1)$ (Lemma~\ref{lem-necessary}).
  \item There exists $c\approx 1.83297$ such that
    $t_{N,p}\leq \frac{cN}{2}\log(p-1)+\epsilon_N$ for all $p\geq 4$, with
    $\epsilon_N \to 0$ \cite[Theorem 2.6]{FrHoLeUlZh17}.

  \item The above $c$ can be sharpened to $\frac{9}{8}\log(2)+1\approx 1.7798$
    for all $p\geq 4$ (Theorem \ref{thm2}) .
  \item Let $c_p$ be the best constant $c$ for fixed $p\geq 4$. Then
    $ \lim_{p\rightarrow \infty}c_p=1$ (Theorem \ref{thm1}).
  \item The constant $c$ can be chosen to be $1$ under the additional condition
    that $h(fu_{i,j})=0$ for all $1\leq i,j\leq N$ (Theorem \ref{thm3}) .
  \end{enumerate}

  In the case of duals of discrete groups we have $t_{p} = \frac{1}{2}\log(p-1)$
  for the Poisson semigroup on $\tor^N$ (Weissler and Bonami's induction trick),
  on the dual of $\z_2^{*N}$ \cite{JuPaPaPeRi15} and on the dual of
  $\mathbb{F}_N=\z^{*N}$ \cite{JuPaPaPe17,RiXu16}.
\end{remark}

\bibliographystyle{alpha} \bibliography{BVY}
    
\end{document}